\documentclass[11pt]{amsart}
\usepackage{amsmath,amsthm,verbatim,amssymb,amsfonts,amscd, graphicx, color}
\usepackage{graphics}
\theoremstyle{plain}
\newtheorem{theorem}{Theorem}
\newtheorem{corollary}{Corollary}
\newtheorem{lemma}{Lemma}
\newtheorem{proposition}{Proposition}
\newtheorem{conj}{Conjecture}
\newtheorem*{conj1}{Conjecture~\ref{conj:hard}}

\theoremstyle{remark}
\newtheorem{rem}[theorem]{Remark}

\theoremstyle{definition}

\newcommand{\bin}[2]{
 \left(
   \begin{array}{@{}c@{}}
      #1 \\ #2
  \end{array}
 \right)  }

\newcommand{\Z}{\mathbb{Z}}

\newcommand{\F}{\mathbb{F}}

\title[Heegaard Floer homology of torus links]{Heegaard Floer homology of $(n,n)$-torus links: Computations and Questions}

\author[J. Licata]{Joan E. Licata} \address{Mathematical Sciences Institute, Building 27,
  The Australian National University,  ACT 0200, Australia} \email{joan.licata@anu.edu.au}\thanks{While writing this paper, the author was  partially supported by NSF grant DMS-1237324.  This material is also based upon work supported by the NSF under agreement No. DMS-0635607.}
  
\begin{document}

\maketitle

\begin{abstract}
In this article we study the Heegaard Floer link homology of  $(n, n)$-torus links.  The Alexander multigradings which support non-trivial homology form a string of $n-1$ hypercubes in $\mathbb{Z}^{n}$, and we compute the ranks and gradings of the homology in nearly all Alexander gradings.  We also conjecture a complete description of the link homology and provide some support for this conjecture.  This article is taken from the author's 2007 Ph.D. thesis and contains several open questions.
\end{abstract}

\section{Introduction}
This article focuses on the family of $(n, n)$-torus links, which we denote by $T_n$.  Interest in this family of links stems from two sources.  On the one hand, previous work has shown close relationships between the Heegaard Floer invariants of cabled links and those of the pattern and satellite  \cite{Hom}, \cite{Hed1}, \cite{Hed2}, and this work expands the library of torus link invariants for reference purposes.  On the other hand, the torus links exhibit a great degree of symmetry, both individually and as a family, and these relationships may be exploited to get maximal bang for the computational buck.  In a sense, these links serve to demonstrate  how much information is provided by formal properties of Heegaard Floer invariants.  This article is based on unpublished material from the author's Ph.D. thesis which we present here in part as an invitation to answer some open questions.

We orient  $T_n$ so that if the link were embedded on a torus, all the components would intersect a fixed meridian with the same sign. The Heegaard Floer invariant $\widehat{HFL}(S^3, T_n)$ admits a relative $\mathbb{Z}^{n+1}$ multigrading, with a homological  $\mathbb{Z}$ grading and an n-component Alexander  grading.  Our first theorem gives a geometric description of the absolute Alexander gradings which support non-vanishing homology.

  \begin{theorem}\label{thm:alexsup}
  The Alexander support of $\widehat{HFL}(S^3, T_n; \mathbb{Z}_2)$ consists of unit hypercubes centered at $(\frac{-n}{2}+k,  \dots, \frac{-n}{2}+k)$ for $k\in \{1, 2, \dots, n-1\}$.
   \end{theorem}

More descriptively, this is a string of hypercubes stretched along the main diagonal of $\mathbb{R}^n$.  Starting with the hypercube with the highest Alexander gradings, we label the cubes from $1$ to $n-1$.

\begin{figure}[h!]
\begin{center}
\scalebox{.6}{\includegraphics{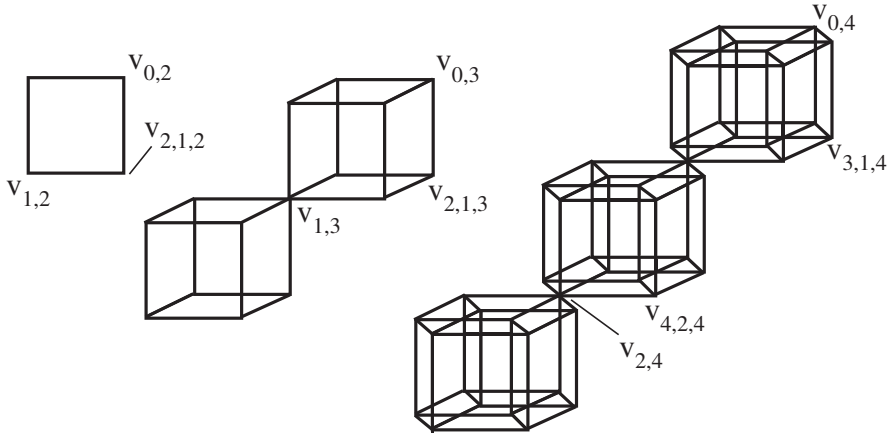}}
\end{center}
\label{HDEx}
\caption{The Alexander support for $\widehat{HFL}(T_2)$, $\widehat{HFL}(T_3)$, and $\widehat{HFL}(T_4)$.  Selected vertices are labeled to illustrate the conventions described immediately before Theorem~\ref{thm:vhom}.}\label {fig:sup}
\end{figure}

Each lattice point on a fixed hyperplane orthogonal to the main diagonal supports identical Heegaard Floer groups.  We label the slices in each hypercube from $1$ to $n+1$, starting with the slice with the highest Alexander gradings.  Notice that the highest-numbered slice of one hypercube is a single lattice point which coincides with the lowest-ordered slice of the next hypercube.  (For example, consider the vertices labeled $v_{1,3}$ and $v_{2,4}$ in  Figure~\ref{fig:sup}.) We introduce separate notation for the extremal and interior slices of the hypercubes.  

Define  $v_{s, c, n}$ to be a lattice point in the $s^{th}$ hyperplane slice of the $c^{th}$ hypercube for $T_n$.  Define $v_{C,n}$ to be the lattice point where the $C^{th}$ hypercube meets the $(C+1)^{th}$ hypercube.

\begin{theorem}\label{thm:vhom}
For $2 \leq s \leq n$, $\widehat{HFL}(T_n, v_{s,c,n};\mathbb{Z}_2)$ has rank $\bin{n-2}{c-1}$ and is supported in homological degree $-c^2-s+2$. Furthermore, $\widehat{HFL}(T_n, v_{0,n};\mathbb{Z}_2)$ has rank $1$ and is supported in degree $0$ and $\widehat{HFL}(T_n, v_{n-1, n};\mathbb{Z}_2)$ has rank $1$ and is supported in degree $n-n^2$.   
\end{theorem}

The proof of this theorem relies on straightforward calculations using a particularly tractable Heegaard diagram for the link. It doesn't address the behavior of the homology groups at the lattice points where two hypercubes meet, however, and we can unfortunately offer only a conjectural answer to this question.  

Let $r{\F}_g$ denote a rank $r$ vector space with homological grading $g$. 

\begin{conj}\label{conj:hard}
 For $1\leq C \le \frac{n-1}{2}$,
\[
\widehat{HFL}(T_n, v_{C,n};\mathbb{Z}_2)\cong
      \bigoplus_{i=0}^C {\bin{n-1}{i}} \F_{-C^2-C-i}\oplus \bigoplus_{j=0}^{C-1} {\bin{n-1}{j}}\F_{-C^2-C-n+2+j} .
\]
\end{conj}

Assuming the conjecture,  symmetry of the link invariant implies the following:

\begin{corollary} 
For $\frac{n-1}{2}< C<n-1$
\[
\widehat{HFL}(T_n, v_{C,n})\cong \widehat{HFL}(T_n, v_{n-1-C,n})[ n^2-n-2Cn] 
\]  where $[\cdot]$ denotes a shift in the grading by the indicated amount.
\end{corollary}

Conjecture~\ref{conj:hard} has been verified by direct computation for $n\leq 6$.   More persuasive support comes from the following theorem, which demonstrates that for all $C$, the behavior at the extremal vertex of the $C^{th}$ hypercube stabilizes for large values of $n$:

\begin{theorem}\label{thm:hard}
For sufficiently large $n$, if Conjecture~\ref{conj:hard} holds at $v_{C, n-1}$ and $v_{C-1, n-1}$, then it holds for $v_{C,n}$, as well.  
\end{theorem}

Where the proof of Theorem~\ref{thm:vhom} relied on the symmetry within a single $T_n$ link, the proof of Theorem~\ref{thm:hard} depends on the relationship between torus links for different values of $n$.  

This approach yields analogous results for the link homology of $T(n, sn)$; as in the $s=1$ case, the Alexander gradings supporting non-trivial homology form a string of  $s(n-1)$ hypercubes in ${\Z}^{n}$.  Determining ranks and gradings at hyperplane slices in the interior of each cube is similarly straightforward, but the homology at lattice points where two hypercubes join is again more complicated.  

We end this section with some speculation.  

\subsubsection{Representations of the symmetric group}

For any set of components of $T_n$, one may find an isotopy which permutes the components while fixing the link setwise.  We conjecture that this induces an action of the symmetric group $S_n$ on the Heegaard Floer link homology in the following sense.

Fix a Heegaard diagram $\mathcal{D}_1=(\Sigma, \mathbf{\alpha}, \mathbf{\beta}, z_1, \dots, z_n, w_1, \dots, w_n)$ for $T_n\subset S^3$ and a permutation $\sigma\in S_n$.  Let $\mathcal{D}_2$ denote the Heegaard diagram $(\Sigma, \mathbf{\alpha}, \mathbf{\beta}, z_{\sigma(1)}, \dots, z_{\sigma(n)}, w_{\sigma(1)}, \dots, w_{\sigma(n)})$.  To the diagram $\mathcal{D}_i$ one associates the chain complex $\widehat{CFL}(\mathcal{D}_i)$ and its homology $\widehat{HFL}(\mathcal{D}_i)$;  this notation is chosen to emphasize the specific graded vector space, rather than simply its isomorphism class.   Since there is an isotopy of $T_n$ which sends the components labeled $(1, \dots, n)$ to the components labeled $\big(\sigma(1), \dots, \sigma(n)\big)$, it follows that there exist diagram moves (i.e., isotopies, handleslides, stabilizations) which transform $\mathcal{D}_1$ to $\mathcal{D}_2$.  Fix such a sequence of diagram moves and consider the induced isomorphism of graded vector spaces  $F_\sigma: \widehat{HFL}(\mathcal{D}_1)\rightarrow \widehat{HFL}(\mathcal{D}_2)$.  

On the other hand, there is a canonical isomorphism of the chain complexes $\widehat{CFL}(\mathcal{D}_1)\cong\widehat{CFL}(\mathcal{D}_2)$ which comes from identifying intersection points on the two diagrams.  Denote the induced isomorphism on homology by $G_\sigma: \widehat{HFL}(\mathcal{D}_2)\rightarrow\widehat{HFL}(\mathcal{D}_1)$. Although this isomorphism respects homological degree, it permutes the coordinates of the Alexander grading.  

\begin{conj} The map which sends $\sigma$ to $G_\sigma \circ F_\sigma \in \text{Aut}\big(\widehat{HFL}(\mathcal{D}_1)\big)$ defines a representation of $S_n$ on $\widehat{HFL}(\mathcal{D}_1)$. 
\end{conj}

Assuming the above conjecture, it would be interesting to explicitly identify the resulting representation. Although $G_\sigma$ does not respect the Alexander grading in general, it does preserve the Alexander grading of generators at the extremal vertices of each hypercube.  Thus the homology at each such vertex would be a representation of $S_n$ in its own right.  Conjecture~\ref{conj:hard} asserts that at Alexander gradings $v_{C,n}$, the dimensions of the graded summands are binomial coefficients of the form $\bin{n-1}{k}$; when working over a field of characteristic zero, these are exactly the dimensions of the irreducible representations of $S_n$ which correspond to hook partitions of $n$. However, when working over a field of characteristic two, decomposing the representations of $S_n$ is more subtle.  More speculatively, we propose the following: 

\begin{conj} When the Heegaard Floer link homology is taken over a field of characteristic zero, each graded summand of $\widehat{HFL}(S^3, T_n, v_{C,n})$ is an irreducible representation of $S_n$.
\end{conj}

If proven, these conjectures offer an alternative route to proving a characteristic zero version of Theorem~\ref{thm:hard}. One might hope to interpret the isomorphisms $G_\sigma\circ F_\sigma$ in the broader context of a functorial theory for isotopies or more general cobordisms, but currently it is not even clear how to choose the diagram moves so that these maps satisfy the necessary symmetric group relations.    Nevertheless, actions of the mapping class group of a manifold on  Heegaard Floer groups have been investigated in several specific contexts.  For example, Ozsv\'ath and Stipsicz show that the subgroup of the mapping class group of $(S^3, K)$ which fixes $K$ pointwise induces an action on $HFK^-(K)$ \cite{OStip}.  Similarly, Juh\'asz proved that special cobordisms of sutured manifolds induce a functorial map on sutured Heegaard Floer groups; this specializes to a map on the link homology which is again sensitive to a collection of marked points on the link components  \cite{Juh}.  At present little is known about generalizations to the case when $L$ is merely fixed setwise.

 \subsubsection{A generalized surgery exact sequence}
The surgery exact triangle for knot Floer homology is one of the most important properties of the invariant. The vertices of the triangle correspond to the Heegaard Floer groups of three links $L_-$, $L_0$, and $L_+$ which differ by oriented skein moves at a fixed crossing.  The relationship among the link homologies is a consequence of the surgery exact triangle for manifolds, together with an identification of $\widehat{HFK}(S^2\times S^1, L')$ with $\widehat{HFK}(S^3, L_0)$; here $L'$ is the link in  $S^2\times S^1$ which results from $0$-surgery on the boundary of a disc punctured  algebraically zero times and geometrically twice by the original link $L$.  

One may also consider the exact triangle associated to surgery along the boundary of a disc punctured algebraically zero times and geometrically $2n$ times.  The homological relationship between the $L_-$, $L_+$, and $L'$ chain complexes  carries over to this context, but one can no longer identify the link $L'\subset S^2\times S^1$ with the resolved link $L_0$.  It would therefore be interesting to know whether the $\widehat{HFK}(S^2\times S^1, L')$ term could be identified with some other well-understood object which might play a role analogous to that of the resolved link in the skein exact triangle.  Such a characterization would be helpful not only for computations of the sort addressed in this paper, but also for the more general case of altering a link by introducing a full twist among a collection of parallel strands.

\section{The chain complex $\widehat{CFL}(T_n)$}

\subsection{Preliminaries}\label{sect:prelim}
We assume the reader is familiar with basic definitions from Heegaard Floer theory as found in \cite{OSz2004a} and \cite{OSz2004b}, and in this section we recall a few properties of the link invariant that will be used subsequently. We work exclusively with $\mathbb{Z}_2$ coefficients and the hat theory.

Given the $n$-component link $T_n\subset S^3$, one may construct a spherical Heegaard diagram with $n-1$ of each of the following: $\alpha$ curves, $\beta$ curves, $z$ basepoints, and $w$ basepoints.  To each domain $\phi$ connecting generators, assign the vector $\big(n_{z_1}(\phi)-n_{w_1}(\phi), \dots, n_{z_n}(\phi)-n_{w_n}(\phi)\big)$ which counts intersections with the basepoints; this defines a relative  $\mathbb{Z}^n$ filtration on the chain complex associated to the diagram, and  the link invariant $\widehat{HFL}(S^3, T_n)$ is the homology of the associated graded complex.  Such a diagram  may also be viewed as specifying the knot $K\in S^3\#^{n-1}(S^2\times S^1)$ which is formed by taking the connected sum of the link components along three-dimensional one-handles attached at $(z_i, w_j)$ pairs.  It follows that there is a spectral sequence whose $E^1$ page is $\widehat{HFL}(S^3, T_n)$ and whose $E_\infty$ page is $\widehat{HF}\big(S^3\#^{n-1}(S^2\times S^1)\big)$.  This spectral sequence has several incarnations in the current work: ignoring all the $z$ basepoints yields a Heegaard diagram for $S^3 \#^{n-1}(S^2\times S^1)$, whereas ignoring a single $z_i$ basepoint  yields a diagram for $T_{n-1}$ in $S^3\# (S^2\times S^1)$.  

By convention, the homological grading on $\widehat{HFL}(S^3,L)$ is fixed so that the top-dimensional class in the total homology of the filtered complex $\widehat{HF}\big(S^3\#^{n-1}(S^2\times S^1)\big)$ is supported in degree zero \cite{OSz2008}.  In addition, the link invariant satisfies the following relationship between the homological and Alexander gradings:

\begin{equation}\label{eq:grshift}
\widehat{HFL}_*(T_n, \textbf{v}) \cong \widehat{HFL}_{*-2\delta( \textbf{v})}(T_n, - \textbf{v}).
\end{equation}
Here $ \textbf{v}\in \mathbb{Z}^n$ denotes a fixed Alexander grading, and $\delta( \textbf{v})$ is the sum of the Alexander grading coordinates \cite{OSz2008}.

\subsection{The Heegaard diagram}

We begin by defining a special Heegaard diagram and discussing some properties of the associated graded chain complex.

Consider the spherical Heegaard diagram for $T_n$ shown in Figure~\ref{fig:tnhd}.  We say that an intersection point is \textit{pure} if it has the form $\alpha_i\cap \beta_i$; otherwise, the intersection point is \textit{mixed}. Pure intersection points live either in the \textit{interior}, \textit{exterior}, or \textit{grid} regions of the diagram.  (See Figure~\ref{fig:lab}.)

\begin{figure}[h!]
\begin{center}
\scalebox{.5}{\includegraphics{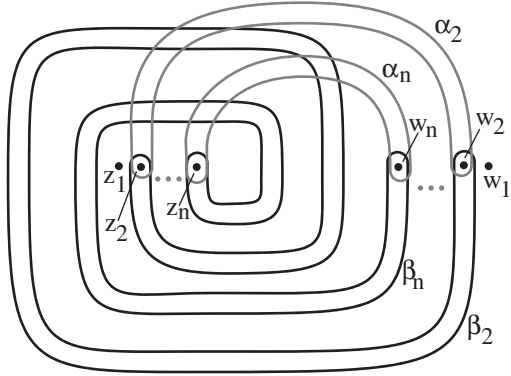}}
\end{center}
\label{HDEx}
\caption{A Heegaard diagram for $T_n$, showing two of the $n-1$ $\alpha$ and $\beta$ pairs.}\label{fig:tnhd}
\end{figure}

\begin{lemma}\label{lem:squares} Every  generator with $k$ mixed intersection points has the same Alexander grading as some  generator composed only of pure intersection points,  at least $k$ of which are in the grid region. 
\end{lemma}

\begin{proof} A generator with $k$ mixed intersection points may be transformed into a pure generator with $k$ pure intersection points in the grid region via a sequence of transpositions of the form 
\[(\alpha_i \cap \beta_j), (\alpha_j \cap \beta_k)\rightarrow (\alpha_j \cap \beta_j), (\alpha_i\cap \beta_k).\] 
Each transposition may be realized by a square domain containing no basepoints which connects the pairs of intersection points in the grid region, so the Alexander grading is preserved.   
\end{proof}

It is straightforward to compute the relative Alexander gradings of the pure generators in terms of contributions from the individual intersection points.  For each $i$, the pure intersection point with the highest Alexander grading is the one labeled $E_i$ in Figure~\ref{fig:lab}. For convenience, we assign each $E_i$ the vector $(0, \dots, 0)$.  To a pure intersection point $X_i$, we then assign the vector whose $j^{th}$ coordinate counts the difference in intersection numbers $n_{z_j}(\phi)-n_{w_j}(\phi)$ for a domain $\phi$ connecting $E_i$ to $X_i$ (Figure~\ref{fig:eig}).  This yields the following:
\[E_i\rightarrow (0, \dots, 0) \ \ \  \ \ \  \ \ \ G_i\rightarrow(\underbrace{1, \dots, 1}_{i-1}, \underbrace{0, \dots, 0}_{n-i+1}) \ \ \ \ \ \ \ \ \ \ I_i\rightarrow(1, \dots, 1)).\]

\begin{figure}
\begin{center}
\scalebox{.5}{\includegraphics{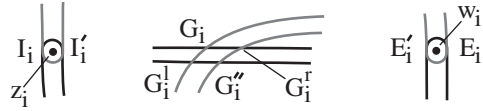}}
\end{center}
\label{HDEx}
\caption{Labels for the inner (I), grid (G), and exterior (E) pure intersection points.}\label{fig:lab} 
\end{figure}

In addition, each superscript ``$\ '\ $", ``$r$", or ``$l$" in the label  of $X_i$ adds a $1$ to the $i^{th}$ coordinate.  For example, if $n=4$, then $G''_2\rightarrow (1,2,0,0)$ and  $I'_4\rightarrow (1,1,1,2)$.  

Using these coordinates, define the working Alexander grading $A$ of a generator $\mathbb{X}$ to be the sum of the  gradings associated to its constituent intersection points.  Thus,  we have  $A(G''_2, E'_3,I'_4)=(2, 3, 2, 2)$ in the chain complex for $T_4$.

\begin{rem} These working coordinates are chosen for convenience, but we note that with this convention, the differential on the chain complex for $S^3\#^{n-1}(S^2\times S^1)$ is grading-increasing.  To recover the absolute Alexander gradings cited in Theorem~\ref{thm:alexsup}, set $\mathbb{E}=(E_2, \dots E_n)$ to have Alexander grading $(\frac{n-1}{2}, \dots, \frac{n-1}{2})$ and apply Equation~\ref{eq:grshift}.  It follows that the absolute Alexander grading of any other generator may be computed by subtracting its working grading from $A(\mathbb{E})$.  We will use the working coordinates throughout the remainder of the paper.
\end{rem}

With this notation in hand, we may assign each generator to a lattice point in $\mathbb{Z}^n$ and describe the geometry of the Alexander support of $\widehat{HFL}(T_n)$.

\begin{lemma}\label{lem:slice} Identify the Alexander gradings of the Heegaard Floer link homology groups with points in $\mathbb{Z}^n$ as described above.  In each hyperplane slice orthogonal to the main diagonal, only the lattice points  lying in the unit hypercube support non-trivial homology.  
\end{lemma}

\begin{proof} 
We claim first that 
\[\widehat{HFL}(T_n,  (x_1, \dots, x_n))\cong \widehat{HFL}(T_n,  (\sigma(x_1), \dots, \sigma(x_n)))\] for any  $\sigma\in S_n$.  To see this, note that any permutation of the components of $T_n$ can be realized by ambient isotopy in $S_3$.  It follows that any permutation of the indices $\{1, \dots, n\}$ of the $(z_j, w_j)$ basepoint pairs yields a Heegaard diagram for an isotopic copy of $T_n$, and hence determines isomorphic  link homology groups.  Since the $j^{th}$ basepoint pair determines the $j^{th}$  coordinate of the Alexander grading, the $\mathbb{Z}$-graded homology groups at any pair of lattice points in the same $S_n$ orbit are isomorphic.
  
The lattice points in a hyperplane slice orthogonal to the main diagonal are characterized by their coordinates having a fixed sum, and a lattice point lies in the $c^{th}$ hypercube exactly when each of its coordinates is either $c$ or $c-1$.  For each lattice point $(x_1, \dots x_n)$ with the property that  some pair of coordinates differs by at least two, we claim there is a lattice point in the same $S_n$ orbit which has no generators.  It immediately follows that the homology must be trivial at every lattice point in the orbit.  

To establish the claim, we turn to the contributions of the pure intersection points computed above.  For a fixed hyperplane slice and orbit disjoint from the unit hypercube, consider the  unique lattice point whose coordinates form a non-decreasing sequence.  The only intersection points whose contributions satisfy $x_n>x_1$ are $E_n'$, $G_n''$ and $I_n'$. In each case $x_n-x_1=1$, and a single generator involves at most one of these intersection points.  However, a non-decreasing sequence in an orbit disjoint from the unit hypercube is characterized by $x_n-x_1\geq 2$, proving the claim.
\end{proof}
    
Let $v_{s,c,n}$ denote the lattice point in the $s^{th}$ hyperplane slice of the $c^{th}$ hypercube whose coordinates are non-decreasing ($2\leq s \leq n$, $1\leq c \leq n-1$):
\[ v_{s,c,n}= (\underbrace{c-1, \dots, c-1}_{n-s+1}, \underbrace{c, \dots, c}_{s-1}).\]

This distinguishes a unique point in each $S_n$ orbit.
    
\begin{lemma}\label{lem:vscn} The lattice point $v_{s,c,n}$ supports exclusively pure generators.  Furthermore, for each of these generators, the only intersection point which could belong to the grid region is $G_n''$. 
\end{lemma}

\begin{proof} 
We begin with the second claim. For $i<n$,  the Alexander grading vector associated to $G_i^*$ begins with a $1$ and ends with a $0$.  As noted in the proof of Lemma~\ref{lem:slice}, the only intersection points whose final coordinates are greater than their initial ones come from $\alpha_n\cap \beta_n$, and in each case, the difference is only $1$.  The only non-decreasing sequence of Alexander coordinates that could be associated to a generator involving some 
 $G_i^*$ for $i<n$ is therefore a constant sequence, in which case it cannot be $v_{s,c,n}$. 

The first part of the lemma now follows from Lemma~\ref{lem:squares}, since it is impossible to have a single mixed intersection point.
\end{proof}

\section{Homological grading}\label{sect:homgr}
In this section we focus on the homological grading.  

\begin{proposition}\label{prop:topgr} The generator $\mathbb{E}=(E_2, \dots, E_n)$ represents a non-trivial homology class in $\widehat{HFL}(S^3, T_n)$ with homological grading $0$. 
\end{proposition}

The proof of this proposition appears at the end of the section.  In the meantime, it will be helpful to take a closer look at the generators at the lattice point $v_{s,c,n}$.  These split into three classes, each characterized by the pure intersection point $x_n$:

\

\[ \mathbb{X}\in \begin{cases} \mathcal{E}_{s,c,n}&\text{ if } x_n=E_n'\\
\mathcal{I}_{s,c,n}&\text{ if } x_n=I_n'\\
\mathcal{G}_{s,c,n}&\text{ if } x_n=G_n''.\end{cases}\]

In each class, the total number of $G$ and $I$ intersection points is $c-1$, and the remaining $n-c$ intersection points are in the $E$ region.  It follows from Lemma~\ref{lem:vscn} that there are $\bin{n-2}{c-1}$ generators in $\mathcal{E}_{s,c,n}$ and $\bin{n-2}{c-2}$ generators in each of $\mathcal{I}_{s,c,n}$ and $\mathcal{G}_{s,c,n}$. The ``s" parameter corresponds to the number of intersection points labeled with a prime.  

\begin{lemma}\label{lem:relgr} The generators of type $\mathcal{E}_{s,c,n}$ and $\mathcal{I}_{s,c,n}$ have homological grading $-c^2-s+2$ and the generators of type $\mathcal{G}_{s,c,n}$ have homological grading $-c^2-s+3$.
\end{lemma}

\begin{proof}
Assuming Proposition~\ref{prop:topgr}, it suffices to compute the relative gradings of $\mathbb{E}$ and the generators with Alexander grading $v_{s,c,n}$.  As usual, relative homological grading is determined by connecting the two generators by a domain and computing its Maslov index.  It is easy to check that adding a prime to the label of any intersection point lowers the homological grading by one, so it suffices to compute the homological gradings of generators in the second slice of each hypercube, i.e.,  for the generators with Alexander grading $v_{2,c,n}$.

 Recall that the Maslov index of a domain $\phi$ connecting $\mathbf{x}$ to $\mathbf{y}$ is computed by the following formula \cite{Lip}:
\begin{align}
\mu(\phi)&=n_{\mathbf{x}}(\phi)+n_{\mathbf{y}}(\phi)+e(\phi)\notag\\
&=n_{\mathbf{x}}(\phi)+ n_{\mathbf{y}}(\phi)-2n_{\mathbf{w}}(\phi)-k/4+l/4+\chi(\phi).\notag
\end{align}

Here, $n_*(\phi)$ denotes the (possibly fractional) local multiplicity of the tuple of points $*$ in $\phi$. The formula for the Euler measure $e(\mathcal{\phi})$ also takes into account the number of acute ($k$) and obtuse ($l$) corners of the domain with Euler characteristic $\chi(\phi)$.

The shaded region in Figure~\ref{fig:eig} shows a domain with two obtuse corners connecting $E_i$ to $G_i$; this intersects $E_j$ for $j>i$,  $w_j$ for $j\geq i$, and $I_j$ for all $j$.  It follows from the formula  above  that the Maslov index of this domain is  $1$.   The darkly shaded region in  Figure~\ref{fig:eig} shows a domain with one obtuse and one acute corner connecting $G_i$ to $I_i$; this intersects $I_j$ for all $j>i$, so the Maslov index of this domain depends on the other intersection points.  Some careful bookkeeping shows that if $\mathbb{X}$ is a generator with $k$ intersection points of type $I$ and the rest of type $E$, then the homological grading of $\mathbb{X}$ is $-k^2-2k$. Setting $k=c-1$ and subtracting one for the prime on the label of $x_n$, we see that all the generators of type $\mathcal{E}_{2,c,n}$ and $\mathcal{I}_{2,c,n}$ have homological grading $-c^2$.  Since each generator in $\mathcal{G}_{2,c,n}$ connects to a generator in $\mathcal{I}_{2,c,n}$ via a bigon disjoint from all the basepoints, the homological grading of each generator in $\mathcal{G}_{2,c,n}$ is $-c^2+1$. 

\end{proof}

 \begin{proof}[Proof of Proposition~\ref{prop:topgr}]
The proof of Lemma~\ref{lem:relgr} shows that $\mathbb{E}$ is the top-graded generator of $\widehat{CFL}(T_n)$.  This generator is unique in its Alexander grading, so it clearly represents a homology class in $\widehat{HFL}(T_n)$. As described in Section~\ref{sect:prelim}, the $z$-basepoint filtration induces a spectral sequence which converges to $\widehat{HF}\big(S^3\#^{n-1}(S^2 \times S^1)\big)$. Consequently, the only circumstance under which $\mathbb{E}$ could fail to represent a homology class in the $E^\infty$ page is if there exists some higher differential in the spectral sequence which doesn't send it to zero.  We claim that no such map can exist, 

The spectral sequence preserves homological gradings, so Lemma~\ref{lem:relgr} implies that any higher differential maps $\mathbb{E}$ to something in the span of the terms coming from the second slice of the first hypercube.  The sum of the Alexander coordinates of such terms differs from $A(\mathbb{E})$ by one, so we may restrict our attention to maps between $E^1$ terms.  

Observe that each lattice point in second slice of the first hypercube supports a unique generator.   Consequently, each of these lattice points supports one-dimensional homology on the $E^1$ page. The symmetry of the homology under the symmetric group action implies that any non-trivial $E^1$ differential from $[\mathbb{E}]$ would be induced by boundary discs in the $\widehat{HF}\big(S^3\#^{n-1}(S^2 \times S^1)\big)$ complex connecting $\mathbb{E}$ to the $n$ generators in the second slice of the first hypercube.    It therefore suffices to show that the differential does not map $\mathbb{E}$ to the generator $(E_2, E_3, \dots, E_{n-1}, E_n')$.  There are two bigon domains connecting this pair of generators, each of which crosses $z_n$ but no $w_i$ basepoints; these contributions to the boundary cancel since we are working with $\mathbb{Z}_2$ coefficients.  Any other domain differs from one of these by the addition of some periodic domain disjoint from the basepoints.  Enumerating these shows that any other domain has some region with negative multiplicity, a feature which precludes it from supporting a holomorphic disc.  We can therefore conclude that $[\mathbb{E}]$ persists in the homology of the $E^\infty$ page, and that it has homological grading zero.
%The point is that by considering the homological grading of the E^1 terms, we can restrict our attention to a small list of discs to consider.  If we wanted to show directly that $\mathbb{E}$ is a cycle in the manifold complex, then we'd have to consider all possible differentials; we suspect that there in fact aren't any generators with the right homological grading outside of this fixed slice, but this way we don't have to address that question.
\end{proof}

\subsection{Homology in the interior of a hypercube}
Theorem~\ref{thm:vhom}  describes the homology at lattice points in the interior of the hypercubes.  Here, we assemble results from the previous section to prove this theorem.

\begin{proof}[Proof of Theorem~\ref{thm:vhom}]
Lemma~\ref{lem:relgr} shows that $v_{s,c,n}$ supports generators in two consecutive homological gradings.  We show that the differential at $v_{s,c,n}$ is injective on the subcomplex spanned by the generators $\mathcal{G}_{s,c,n}$.  The theorem then follows from counting dimensions:
\[ \text{rank } \widehat{HFL}(T_n, v_{s,c,n})= |\mathcal{E}_{s,c,n}|  +|\mathcal{I}_{s,c,n}| -|\mathcal{G}_{s,c,n} |. \]

We claim that the sets $\mathcal{G}_{s,c,n}$ and $\mathcal{I}_{s,c,n}$  admit maps $N_{\mathcal{G}}$ and $N_{\mathcal{I}}$ to $\mathbb{N}$ with the property that the differential is weakly increasing with respect to these functions: if $N_{\mathcal{G}}(\mathbf{x})>N_{\mathcal{I}}(\mathbf{y})$, then $\mathbf{y}$ is not  a summand of $\partial \mathbf{x}$.  To prove this, we describe a model domain connecting $\mathbf{x}$ to $\mathbf{y}$ for each pair of generators in $\mathbf{x}\in\mathcal{G}_{s,c,n}$ and $\mathbf{y}\in\mathcal{I}_{s,c,n}$.  

We identify a point $p_\mathcal{\phi}$ on the Heegaard diagram  with the property that the multiplicity of the model domain $\mathcal{\phi}$ at $w_1$ is greater than its multiplicity at $p_\mathcal{\phi}$.  The point $p_\mathcal{\phi}$ is chosen so that any periodic domain has the same multiplicity at $p_\mathcal{\phi}$ as at $w_1$, so it follows that no domain disjoint from $w_1$ can have non-negative multiplicity in every region, a necessary condition for a domain to support a holomorphic representative.  

In order to define the functions $N_*$, we note that generators in $\mathcal{G}_{s,c,n}$ and $\mathcal{I}_{s,c,n}$ are distinguished by their choice of $I_i^*$ intersection points.  To a generator from either class,  associate a length $n-2$  binary sequence by replacing each $E_i^*$ by a $0$ and each $I_i^*$ by a $1$. 

There is a natural linear ordering on such sequences which is given by setting $S>T$ if, for the least $j$ such that $s_j\neq t_j$, $s_j=1$.  For example, 
\[ (1,1,0,1)>(1,0,1,1)>(0,1,1,1).\]
This ordering induces the maps \[N_{\mathcal{G}}: \mathcal{G}_{s,c,n}\rightarrow \{1, 2, \dots, \bin{n-2}{c-2}\}\text{ and }N_{\mathcal{I}}: \mathcal{I}_{s,c,n}\rightarrow \{1, 2, \dots, \bin{n-2}{c-2}\}.\] 

Observe that $N_{\mathcal{G}}(\mathbf{x})=N_{\mathcal{I}}(\mathbf{y})$ if the two generators differ only by replacing $G_n''$ with $I_n'$. This replacement is realized by a bigon disjoint from all the basepoints, so $\mathbf{y}$ is a summand of $\partial \mathbf{x}$ for all $\mathbf{x}\in  \mathcal{G}_{s,c,n}$.  Together with the fact that the differential is weakly increasing with respect to the $N_*$, which is proved in the next paragraph, this proves that the differential is injective  on the subcomplex spanned by the generators in $\mathcal{G}_{s,c,n}$.

For each $i<n$, Figure~\ref{fig:eig} indicates a domain that connects $E_i$ to $I_i$ and one that connects $I_i$ to $E_i$.  Domains of the first type have multiplicity zero at $w_1$, whereas domains of the second type have multiplicity two at $w_1$; both types of domains have multiplicity one at the point $p_i$ lying just to the left of $\beta_j$ near the exterior intersection region, as indicated in Figure~\ref{fig:eig}.  For each pair of generators $\mathbf{x}\in \mathcal{G}_{s,c,n}$ and $\mathbf{y} \in \mathcal{I}_{s,c,n}$,  define the domain $\phi$ from $\mathbf{x}$ to $\mathbf{y}$ as the union of the model domains associated to each pair of intersection points where they differ.   If the least $j$ for which $x_j\neq y_j$ has $x_j=I_j^*$, pick $p_\mathcal{\phi}=p_j$.   With this choice,  the domain satisfies $0<n_{p_\mathcal{\phi}}(\mathcal{\phi})<n_{w_1}(\mathcal{\phi})$.  As described above, this implies the desired result that there can be no holomorphic disc connecting $\mathbf{x}$ to $\mathbf{y}$.    
\end{proof}

\begin{figure}
\begin{center}
\scalebox{.5}{\includegraphics{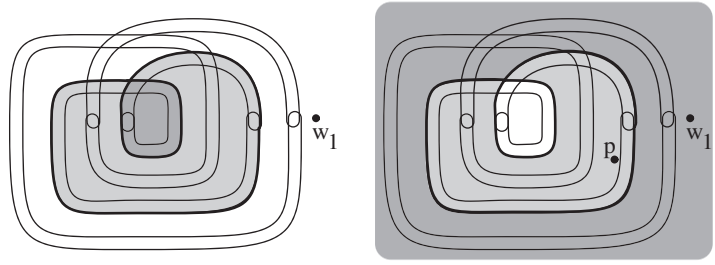}}
\end{center}
\caption{Left: Domain connecting $E_i$ to $I_i$.  Right: Domain connecting $I_i$ to $E_i$. In the lightly shaded areas, the domain has multiplicity one, and in the heavily shaded areas, it has multiplicity two. The diagram on the left shows the two special domains described in the proof of Lemma~\ref{lem:relgr} superimposed.}\label{fig:eig}
\end{figure}

\section{Extremal vertices of hypercubes}\label{sect:vhard}

The computations in the interior of each hypercube are simplified by symmetry considerations which allow one to disregard most generators with intersection points in the grid diagram.  This is no longer the case for Alexander gradings appearing as  extremal vertices of the hypercubes.  Explicit computation of the homology is feasible for the vertices where the first few hypercubes meet, and symmetry of the complex then determines the homology where the last few hypercubes meet.  Nevertheless, a more systematic approach to the homology at the intersection of two hypercubes remains elusive.    Instead, we offer a conjecture about the homology at these lattice points and describe both experimental and inductive support for it.

Denote the lattice point where the $c^{th}$ and $(c+1)^{th}$ hypercubes meet by $v_{C,n}$.

\begin{conj1}
 When $C \le \frac{n}{2}$
\[
\widehat{HFL}(T_n, v_{C,n})\cong
      \bigoplus_{i=0}^C {\bin{n-1}{i}} \F_{-C^2-C-i}\oplus \bigoplus_{j=0}^{C-1} {\bin{n-1}{j}}\F_{-C^2-C-n+2+j} .
\]
\end{conj1}

\subsection{The homology at $v_{1,n}$} 
We prove the following special case of Conjecture~\ref{conj:hard}, in part to illustrate our approach to Theorem~\ref{thm:hard}. As before, let $r{\F}_{g}$ denote rank $r$ homology supported in homological degree $g$.

\begin{proposition}\label{prop:v1}
For $n\geq 3$, 

\[   \widehat{HFL}(T_n,  v_{1,n})\cong \mathbb{F}_{-2}\oplus (n-1) \mathbb{F}_{-3}\oplus \mathbb{F}_{-n}.
\]  
\end{proposition}

As described in Section~\ref{sect:prelim}, each  component of $L$ induces a $\mathbb{Z}$ filtration on the underlying complex for the manifold; ignoring the filtration induced by the component $L_i$ yields a complex which computes $\widehat{HFL}\big(S^3\#(S^2\times S^1), L\setminus L_i\big)$. This complex is filtered chain homotopic to
\[ \widehat{CFL}(S^3, L\setminus L_i) \otimes ({\F}_{0} \oplus {\F}_{-1}).\]

Deleting a single component from $T_n$ leaves the link $T_{n-1}$, so this relationship provides a useful tool for studying the homology of various torus links simultaneously.

Let $f_i:\mathbb{Z}^n\rightarrow \mathbb{Z}^{n-1}$ be the forgetful map that collapses the $i^{th}$ coordinate of the  filtration on $\widehat{CFL}(S^3, T_n)$.  For any lattice point $\mathbf{v}\in \mathbb{Z}^n$, the generators at $f_i(\mathbf{v})$ may be viewed as a filtered subcomplex of $\widehat{CFL}(S^2\times S^1, T_{n-1})$.  Thus, there is a spectral sequence whose $E^1$ terms are $\widehat{HFL}(S^3, T_n, \mathbf{u})$ for $\mathbf{u}\in f^{-1}_i\big(f_i(\mathbf{v})\big)$ and which converges to
 \[\widehat{HFL}\big(S^3, T_{n-1}, f_i(\mathbf{v})\big) \otimes ({\F}_{0} \oplus {\F}_{-1}).\]
We analyze this spectral sequence for $\mathbf{v}=v_{1,n}$ in order to prove Proposition~\ref{prop:v1} and then generalize the argument in the next section to prove Theorem~\ref{thm:hard}.

\begin{proof}
Proposition~\ref{prop:topgr} proved that $\widehat{HFL}(T_n, v_{0,n})$ is supported in degree zero for any $n$, so $f_{i}(v_{0,n})= v_{0, n-1}$.  It follows that the $c^{th}$ hypercube of $\widehat{HFL}(T_n)$ projects to the $c^{th}$ hypercube of $\widehat{HFL}(T_{n-1})$; Figure~\ref{fig:proj} illustrates this relationship schematically.

\begin{figure}[h!]
\begin{center}
\scalebox{.7}{\includegraphics{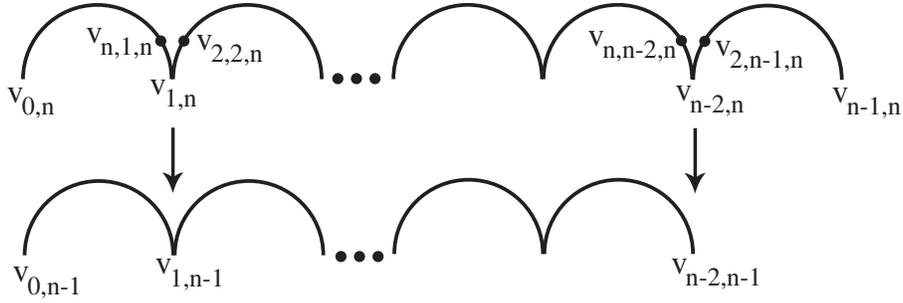}}
\end{center}
\label{HDEx}
\caption{ The forgetful map projects $\widehat{HFL}(T_n)$ to $\widehat{HFL}(T_{n-1})$.  Here, each arc represents a single hypercube, and the working Alexander grading increases to the right; the direction of the filtered boundary map agrees with the direction of the arrows in the short exact sequences below.
}\label {fig:proj}
\end{figure}

According to Equation~\ref{eq:grshift} in Section~\ref{sect:prelim},  the homology at either one of $v_{x,n}$ or $v_{n-1-x,n}$ determines the other.  In the symmetrized coordinates, $\mathbb{E}_n$ has Alexander grading $(\frac{n-1}{2}, \dots, \frac{n-1}{2})$, so Equation~\ref{eq:grshift} implies that  the one-dimensional group $\widehat{HFL}(T_n, v_{n-2, n-1})$ is supported in degree $-n^2+3n-2$.

The map $f_i$ sends the lattice point $v_{n-2, n}$ to  $v_{n-2, n-1}$, and the $f_i$-preimage of $v_{n-2, n-1}$  consists of the three points $v_{n,n-2,n}$, $v_{n-2, n}$, and $v_{2, n-1,n}$.  Consequently, there exists a three-term spectral sequence converging by the $E^3$ page to 
\begin{align}
 \widehat{HFL}&(T_{n-1}, v_{n-2, n-1})\otimes ({\F}_{0}\oplus {\F}_{-1})\notag\\
 &\cong  {\F}_{-n^2+3n-2}\otimes ({\F}_{0}\oplus {\F}_{-1})\notag\\
 &\cong  {\F}_{-n^2+3n-2}\oplus  {\F}_{-n^2+3n-3}.\notag\end{align}
 
 Theorem~\ref{thm:vhom} allows us to make the $E^1$ page of this spectral sequence more explicit:
 \[
 (n-2){\F}_{-n^2+3n-2} \rightarrow \widehat{HFL}(T_n, v_{n-2,n}) \rightarrow {\F}_{-n^2+2n-1} 
 \]
 Thus,
 
         \[\widehat{HFL}(T_n, v_{n-2,n}) =(n-2){\F}_{-n^2+3n-3} \oplus {\F}_{-n^2+2n} \oplus k\big[{\F}_{ -n^2+3n-2} \oplus  {\F}_{-n^2+3n-3 }   \big], 
         \]
         where $k=0$ or $1$.
        
Applying Equation~\ref{eq:grshift} again, this implies 
\begin{equation}\label{eq:a1}
\widehat{HFL}(T_n, v_{1,n}) =(n-2){\F}_{-3} \oplus {\F}_{-n} \oplus k\big[{\F}_{ -3} \oplus  {\F}_{-2}   \big].
\end{equation}

Suppose that the Proposition holds for $v_{1, n-1}$,  $n> 3$.  Then the $E^1$ sequence
\[
{\F}_{1-n} \rightarrow \widehat{HFL}(T_n, v_{1,n}) \rightarrow (n-2){\F}_{-4}
\]

converges to 
\[ \widehat{HFL}(T_{n-1}, v_{1, n-1})\otimes ({\F}_{0}\oplus {\F}_{-1})\cong {\F}_{-2}\oplus {\F}_{1-n} \oplus (n-1){\F}_{-3} \oplus {\F}_{-n}\oplus (n-2){\F}_{-4}.\]

  Thus
\[
\widehat{HFL}(T_n, v_{1,n})= {\F}_{-2}\oplus (n-1){\F}_{-3} \oplus {\F}_{-n} \oplus X.
\]
Although we could be more precise about the possible ranks and gradings for $X$, comparison with  Equation~\ref{eq:a1} shows that in fact, $X=\emptyset$.  This proves that
\[
\widehat{HFL}(T_n, v_{1,n})= {\F}_{-2} \oplus (n-1) {\F}_{-3} \oplus {\F}_{-n}
\] and completes the inductive step.   

For small $n$, the proposition may be verified by direct computation, so we take $n=4$ as a sufficient base case for the induction.  
\end{proof}

\subsection{Proof of Theorem~\ref{thm:hard}}
The proof of Theorem~\ref{thm:hard} generalizes the argument used in the previous section.  We study the pair $\widehat{HFL}(T_n, v_{C,n})$ and $\widehat{HFL}(T_n, v_{n-1-C, n})$ simultaneously, in each case using the spectral sequence which converges to the corresponding terms of $\widehat{HFL}(T_{n-1})$.  

Theorem~\ref{thm:hard} is stated with the ambiguous caveat, ``For sufficiently large $n-1$\dots".  We note that this condition arises from  analyzing  short exact sequences as in the proof of Proposition~\ref{prop:v1}.  In order to rule out certain cancellations, $n$ must be large enough so that the polynomials that appear in the computation do not coincidentally evaluate to consecutive integers;  in the  calculations which follow, we assume such a sufficiently large $n$.  %If the $E^2$ terms were known to be stable, then the conjecture holds for any $v_{C,n}$ as long as it holds for $v_{C,n-1}$ and $v_{C-1,n-1}$.  

\begin{proof}
Fix $C\leq \frac{n-1}{2}$. 
The spectral sequence 
\[ \widehat{HFL}(T_n, v_{n, n-C-1, n})\rightarrow \widehat{HFL}(T_n, v_{n-1-C, n})\rightarrow \widehat{HFL}(T_n, v_{2, n-C, n})   \]  converges to $\widehat{HFL}(T_{n-1}, v_{n-1-C, n-1})$.  Applying Theorem~\ref{thm:vhom}, we rewrite this as
\begin{align}
\bin{n-2}{n-C-2}{\F}_{-C^2-2C+2nC-n^2+n+1} & \rightarrow \widehat{HFL}(T_{n}, v_{n-1-C, n})\label{eq:spec}\\
& \rightarrow \bin {n-2}{n-C-1}{\F}_{-C^2+2nC-n^2}.\notag
\end{align}

By hypothesis, Conjecture~\ref{conj:hard} holds for $v_{C-1, n-1}$ so Equation~\ref{eq:grshift} implies   
\begin{align}
\widehat{HFL}&(T_{n-1}, v_{n-1-C, n-1}) \notag\\
&\cong 
\bigoplus_{i=0}^{C-1}\bin{n-2}{i} {\F}_{-C^2-C-n^2+n+2nC-i}\ \oplus 
 \bigoplus_{j=0}^{C-2} \bin{n-2}{j} {\F}_{-C^2-C-n^2+2nC+3+j}.\notag
\end{align}

We note the elementary identity $\bin{n-2}{i}+\bin{n-2}{i+1}=\bin{n-1}{i+1}$.  Thus the spectral sequence in Line~\ref{eq:spec}  converges to 
\begin{align}\label{eq:projhom}
\bigoplus_{i=0}^{C-1} \bin{n-1}{i} {\F}_{-C^2-C+2nC-n^2+n-i} \oplus \bigoplus_{j=0}^{C-2} \bin{n-1}{j} {\F}_{-C^2-C+2nC-n^2+2+j}\notag\\
 \oplus \bin{n-2}{C-1}{\F}_{-C^2-2C+2nC-n^2+n} \oplus \bin{n-2}{C-2} {\F}_{-C^2-n^2+2nC+1}.
\end{align}

Lines~\ref{eq:spec} and \ref{eq:projhom} imply that the the homology at $v_{n-1-C, n}$ consists of at least the terms in Line~\ref{eq:projhom}.  If the $E^2$ terms were known to be stable, then in fact $\widehat{HFL}(T_n, v_{n-1-C, n})$ would contain the summands predicted by Conjecture~\ref{conj:hard}: 
\[
\bigoplus_{i=0}^{C} \bin{n-1}{i} {\F}_{-C^2-C+2nC-n^2+n-i} \oplus \bigoplus_{i=0}^{C-1} \bin{n-1}{j} {\F}_{-C^2-C+2nC-n^2+2+j}.
\]

This implies that an upper bound for the homology at $v_{C,n}$ is given by
\begin{equation}\label{eq:bnds}
\bigoplus_{i=0}^{C} \bin{n-1}{i} {\F}_{-C^2-C-i} \oplus \bigoplus_{j=0}^{C-1} \bin{n-1}{i} {\F}_{-C^2-C-n+2+j}.
\end{equation}
Note that this is the homology predicted by Conjecture~\ref{conj:hard}.

Repeating the same procedure with $v_{C,n}$ and $v_{C,n-1}$ shows that the exact sequence
\begin{equation}\label{eq:seq}
\bin{n-2}{C-1}{\F}_{-C^2-n+2} \rightarrow \widehat{HFL}(T_n, v_{C,n}) \rightarrow \bin{n-2}{C} {\F}_{-C^2-2C-1}
\end{equation}
converges by the $E^3$ page to
\begin{align}
\bigoplus_{i=0}^C \bin{n-1}{i} {\F}_{-C^2-C-i} \oplus \bigoplus_{j=0}^{C-1} \bin{n-1}{j} {\F}_{-C^2-C-n+2+j}\notag\\
\oplus \bin{n-2}{C}{\F}_{-C^2-2C-1} \oplus \bin{n-2}{C-1}{\F}_{-C^2-n+2}.\notag\end{align}

It is immediately clear that $\widehat{HFL}(T_n, v_{C,n})$ must have at least the first two terms as summands, and in fact, this matches exactly with the maximum possible homology allowed by the bounds from  Line~\ref{eq:bnds}.  The remaining two summands agree precisely with the extreme terms of the short exact sequence in Line~\ref{eq:seq}.  Thus, $\widehat{HFL}(T_n, v_{C,n})$ satisfies Conjecture~\ref{conj:hard}.
    
    \end{proof}

\bibliographystyle{amsplain} 
\bibliography{torus}

\end{document}